\documentclass[12pt]{amsart}
\usepackage[utf8]{inputenc}

\usepackage{amssymb,amsthm,amsfonts,enumerate, epsfig,epstopdf,caption, graphicx, xcolor}
\usepackage{amsmath}
\usepackage{subcaption}
\usepackage{hyperref}

\title[Seifert cobordisms and the volume conjecture]{Seifert cobordisms and the Chen-Yang volume conjecture}
\author{R. Detcherry, E.  Kalfagianni and S. Marasinghe}
\address[]{Université Bourgogne Europe, CNRS, IMB UMR 5584, F-21000 Dijon, France}
\email{renaud.detcherry@u-bourgogne.fr}
\address[]{Department of Mathematics, Michigan State University, East
Lansing, MI, 48824, USA}
\email[]{kalfagia@msu.edu}

\address[]{Department of Mathematics, Michigan State University, East
Lansing, MI, 48824, USA}
\email[]{marasin1@msu.edu}

\begin{document}

\begin{abstract} We study the large-$r$ asymptotic behavior of the Turaev-Viro  invariants   $TV_r(M; e^{\frac{2\pi i}{r}})$ of 3-manifolds with toroidal boundary, under the operation of gluing a Seifert-fibered 3-manifold along a component of $\partial M$. We show that the Turaev-Viro invariants volume conjecture is closed under this operation. As an application we prove the volume conjecture for all Seifert fibered 3-manifolds with boundary and for large classes of graph 3-manifolds.
\smallskip

\smallskip
\end{abstract}

\maketitle

\newtheorem{innercustomgeneric}{\customgenericname}
\providecommand{\customgenericname}{}
\newcommand{\newcustomtheorem}[2]{%
  \newenvironment{#1}[1]
  {%
   \renewcommand\customgenericname{#2}%
   \renewcommand\theinnercustomgeneric{##1}%
   \innercustomgeneric
  } {\endinnercustomgeneric} }

\newcustomtheorem{customthm}{Theorem}
\newcustomtheorem{customlemma}{Lemma}
\newcustomtheorem{customprop}{Proposition}
\newcustomtheorem{customconjecture}{Conjecture}
\newcustomtheorem{customcor}{Corollary}
\newcommand{\Q}{{\mathbb{Q}}}
\newcommand{\R}{{\mathbb{R}}}
\newcommand{\Z}{{\mathbb{Z}}}
\newcommand{\N}{{\mathbb{N}}}
\newcommand{\C}{{\mathbb{C}}}

\theoremstyle{plain}
\newtheorem*{ack*}{Acknowledgements}
\newtheorem{thm}{Theorem}[section]
\newtheorem{lem}[thm]{Lemma}
\newtheorem{prop}[thm]{Proposition}
\newtheorem{defin}[thm]{Definition}
\newtheorem{cor}[thm]{Corollary}
\newtheorem{predefinition}[thm]{Definition}
\newtheorem{conjecture}[thm]{Conjecture}
\newtheorem{preremark}[thm]{Remark}
\newtheorem{pro}[thm]{Problem}
\newenvironment{remark}%
 {\begin{preremark}\upshape}{\end{preremark}} \newenvironment{definition}%
  {\begin{predefinition}\upshape}{\end{predefinition}}

\newtheorem{ex}[thm]{Example}
\newtheorem{ques}[thm]{Question}

\newtheorem*{namedtheorem}{\theoremname}
\newcommand{\theoremname}{testing}
\newenvironment{named}[1]{\renewcommand{\theoremname}{#1}\begin{namedtheorem}}{\end{namedtheorem}}
\newcommand{\effie}[1]{{\color{blue}#1}}

\newcommand{\joe}[1]{{\color{red}#1}}

\section{Introduction}
For a compact 3-manifold $M$, the Turaev-Viro invariants  $TV_r(M) :=TV_r(M; e^{\frac{\pi i}{r}})$ are
 a family of real-valued
 invariants depending on  an odd integer $r\geq 3$.
The  invariants were originally constructed via state sums on triangulations of 3-manifolds \cite{TuraevViro}
 and were later related  to skein-theoretic quantum constructions of
Reshetikhin-Turaev  invariants \cite{ growth6j, BePe, colJvolDKY, Roberts}. In this paper we will follow this viewpoint.
We will view  $TV_r(M)$ through its relation to the  skein-theoretic $\mathrm{SO}_3$-TQFT  as constructed 
by Blanchet, Habegger, Masbaum and Vogel  \cite{BHMV1, BHMV2}.

All the 3-manifolds considered in this paper will be orientable and either closed or with boundary consisting of tori (i.e. toroidal boundary).
We prove the following:

\begin{thm}\label{seifertintro} Let $S$ be a Seifert fibered 3-manifold with at least two boundary components and let $M$ be any 3-manifold  with toroidal boundary.
Then, for any 3-manifold $M'$ obtained by gluing $S$  along a component of $\partial S$ to a component of $\partial M$,
there exist constants $A$ and $K>0$, and a finite set of integers $I$, such that
$$\frac{r^{-K}}{A} TV_r(M) \leqslant TV_r(M') \leqslant Ar^K TV_r(M),$$
for all odd $r$ not divisible by any $p\in I$.
\end{thm}

Some of the most prominent open problems in quantum topology are the volume conjectures, asserting that geometric invariants of 3-manifolds (e.g. hyperbolic volume) are determined
by quantum invariants. Theorem \ref{seifertintro} has applications to the Turaev-Viro  invariants volume conjecture.  The conjecture, which is a natural 3-manifold generalization of the well known
Kashaev \cite{procvol} conjecture, asserts that the large-$r$ asymptotics of the Turaev-Viro  invariants determine the simplicial volume  of 3-manifolds.
Specifically, Chen and Yang \cite{chenyang2018vol} conjectured that for hyperbolic manifolds of finite volume, the
growth rate of the Turaev-Viro invariants is exponential and it determines the hyperbolic volume
of the manifolds. 
In \cite{colJvolDKY},  Detcherry, Kalfagianni and Yang approached this conjecture
 by relating the Turaev-Viro invariants of a  link complement  in $S^3$ to the  colored Jones polynomial of the link. In the process they also conjectured that the asymptotics of values of the colored
Jones  polynomial at  $e^{\frac{2\pi i}{N+\frac{1}{2}}}$ ($N\in \N$)  also determine the volume of hyperbolic links \cite[Question  1.7]{colJvolDKY}.
The  hyperbolic links in $S^3$ for which this
conjecture as well as the Chen-Yang volume conjecture for the Turaev-Viro invariants
have been proved include  the figure-eight knot  \cite{colJvolDKY}, with the full asymptotic expansion in this case given in \cite{Ka-Ho},
 the Whitehead chain links  \cite{WongWhitehead},
and large families of octahedral links, including the octahedral augmented links \cite{Ku, WongYangrelRT, Jessica}.
In \cite{CZ}  Chen and Zhu studied the asymptotics of values of  the colored Jones polynomial of twist knots at $e^{\frac{2\pi i}{N+\frac{1}{2}}}$
and proved the conjecture of
 \cite{colJvolDKY} in this case.
In \cite{CZ2}, the same authors studied the asymptotics of values of  the colored Jones polynomial of twist knots  at $e^{\frac{2\pi i}{N+\frac{1}{M}}}$ ($N, M\in \N$)
 and proved that they also  recover their hyperbolic volume.
 This result was also used by them to prove the original
Kashaev-Murakami-Murakami volume conjecture for twist knots \cite{CZ2}.

By the geometrization theorem any   3-manifold $M$ 
with empty or toroidal boundary admits a canonical decomposition along essential spheres and tori, into pieces that are either hyperbolic or Seifert fibered spaces.
We will refer to this as the geometric decomposition of $M$.
The simplicial volume $\mathrm{Vol}(M)$ of  $M$ is defined as the sum of the volumes of the hyperbolic pieces in this decomposition and it  is equal to its Gromov norm times $v_3\approx 1.01494$ \cite{ThurstonGT3manifolds}, which is the hyperbolic volume of a regular ideal hyperbolic tetrahedron. The simplicial volume is additive under  disjoint unions and connected sums of 3-manifolds as well as  under gluing along essential tori. 
The following generalization of the Chen-Yang Conjecture was stated in  \cite{DKGromov}.

\begin{conjecture}\label{TVvolumeconj}
For every compact orientable  3-manifold $M$, with empty or toroidal boundary, we have
$$LTV(M):=\underset{r \rightarrow \infty, \ r \ \textrm{odd}}{\limsup} \frac {2\pi}  {r} \log |TV_r(M)|=\mathrm{Vol}(M),$$
where $r$ runs over all odd integers.
\end{conjecture}

For hyperbolic $M$ the Chen-Yang conjecture is stated with  $LTV(M)$  being the actual limit  rather than the ``lim sup".
 We note however that the restriction to ``${\limsup}$"  for general  3-manifolds is necessary. Indeed for  Seifert fibered spaces 
the invariants $TV_r(M; e^{\frac{2\pi i}{r}})$ can vanish for infinitely many integers $r$. See, for example, \cite{Shashini}.
Nevertheless,  Conjecture \ref{TVvolumeconj} was shown 
to hold in the form where $LTV(M)$ is defined using the ``limit"   instead of  the ``limsup"  for cases of
non-hyperbolic 3-manifolds also.  The $(2,1)$ cable of the figure-eight knot is such an example of  a cusped 3-manifold \cite{2-cable}.
 An infinite family of such  closed  non-hyperbolic manifolds is constructed in
 \cite{KuM1}.

In  \cite{colJvolDKY} Conjecture \ref{TVvolumeconj} was proved for links in $S^3$ with zero  simplicial volume. In \cite{KuM2} Kumar and Melby showed
that the conjecture is closed under cabling operations of knots in $S^3$.
Applying Theorem \ref{seifertintro}  to complements of hyperbolic links for which Conjecture \ref{TVvolumeconj}  is known we can generalize the results of
\cite{KuM2} and produce new families of satellite links satisfying Conjecture \ref{TVvolumeconj}. 
As an example we state the following:

\begin{cor}\label{f8}If $L$ is a link obtained as an  iterated  satellite of the figure-eight with patterns  torus links, then
$$\textit {LTV}(S^3\setminus L) =\mathrm{Vol}(S^3\setminus L) \approx 2.0298832.$$
\end{cor}

 In the study of Turaev-Viro invariants for links in general 3-manifolds, the colored Jones polynomial is replaced by the relative Reshetikhin-Turaev  invariants \cite[Proposition 5.3]{growth6j}.
Using this approach,  Belletti, Detcherry, Kalfagianni and Yang proved Conjecture  \ref{TVvolumeconj} 
 for fundamental shadow links \cite{growth6j},
which form a class of hyperbolic links in connected sums of $S^1\times S^2$ that gives all orientable 3-manifolds that are either closed or with toroidal boundary  by Dehn filling. In  \cite{Bel}, Belletti
proved the conjecture
for additional families of octahedral hyperbolic links in  connected sums of $S^1\times S^2$ \cite{Bel}.

Generalized volume conjectures for relative Reshetikhin-Turaev  and Turaev-Viro invariants of 3-manifolds with \emph{cone hyperbolic metrics} were stated and studied by  Wong and Yang \cite{WongYangrelRT, WY1} who also pointed out that the volume conjectures of  \cite{chenyang2018vol}  can be studied in this generalized setting.
For more details the reader is referred to the survey article  \cite{Survey}.  Chen and Zhu \cite{CZ1} proved these generalized volume conjectures 
for relative Reshetikhin-Turaev invariants of 3-manifolds obtained  by rational Dehn filling along one component of the Whitehead link. Then, as in  \cite{growth6j}, they proved  the Chen-Yang volume conjecture
for Turaev-Viro invariants of these manifolds, which include the complements of twist knots.

Theorem \ref{seifertintro}  implies that if $\mathrm{Vol}(M)=0$, then
we have $LTV(M)=LTV(M')$.  As a corollary we have the following result, which proves Conjecture  \ref{TVvolumeconj}  for Seifert 3-manifolds that either have  non-empty boundary or  they admit an orientation-reversing involution. 
 
\begin{cor}\label{VCS}  Suppose that $S$ is  an oriented Seifert fibered 3-manifold that either has a non-empty boundary, or it is closed and admits an orientation-reversing involution. Then we have
$$LTV(S)=\underset{r \rightarrow \infty, \ r \ \textrm{odd}}{\limsup} \frac {2\pi}  {r} \log |TV_r(S)|=\mathrm{Vol}(S)=0.$$
\end{cor}

Theorem \ref{seifertintro} generalizes to  large families of 3-manifolds obtained by gluing together Seifert fibered 3-manifolds 
(Corollary \ref{graphglue}). As a result in Corollary \ref{Graph} we also verify Conjecture \ref{TVvolumeconj} for these manifolds.
We note that if $M$ satisfies Conjecture \ref{TVvolumeconj} and the
``lim sup" in the definition of $LTV(M)$ 
is actually a limit,  then Theorem \ref{seifertintro} implies that $M'$ also satisfies the conjecture.
Moreover, a similar statement applies to the  exponential growth conjecture, which is a weaker version of the Chen-Yang volume conjecture and was studied in \cite{DKAdvances, DKIndiana, KMe}. For additional relations of the $SU(2)$- Reshetikhin-Turaev TQFT theory and hyperbolic geometry in dimensions two and three see \cite{AJM}.

The proofs of the results in above references rely at least partially on analytic estimates and direct analysis  of the asymptotics of the Reshetikhin-Turaev and Turaev -Viro invariants.  Such arguments are also used in \cite{Takata}  to study the asymptotic expansion of the
Reshetikhin Turaev invariants of lens spaces.
In contrast, our proofs in this paper rely heavily on TQFT properties and 3-manifold topology.
 In the process, we discuss
an approach that could potentially lead to new progress towards understanding the behavior of the asymptotics of the Turaev-Viro invariants under hyperbolic Dehn filling. For details,
the reader is referred to Section six.


\section{TQFT and Turaev-Viro invariants}
\label{sec:prelim}
In this section we recall how to obtain the Turaev-Viro invariants
from the Reshetikhin-Turaev $\mathrm{SO}_3$-TQFT of \cite{ReTu}.  We begin by summarizing some basic features of the $\mathrm{SO}_3$-TQFT
following the skein-theoretic framework of   \cite{BHMV1, BHMV2}.

\subsection{Preliminaries} For an odd integer $r\geqslant 3$ and a primitive $2r$-th root of unity $q$,  the $\mathrm{SO}_3$-TQFT  functor, denoted by  $RT_r$, 
associates a finite-dimensional  Hermitian $\mathbb{C}$-vector space $RT_r(\Sigma)$, to any closed  oriented surface $\Sigma$, such that:

\begin{enumerate} [(a)]
\item For a disjoint union of two surfaces  $\Sigma \coprod \Sigma'$ we have  $RT_r(\Sigma \coprod \Sigma')=RT_r(\Sigma) \otimes RT_r(\Sigma').$
\item For a closed oriented  $3$-manifold $M$, the value $RT_r(M)\in \mathbb{C}$ is the $\mathrm{SO}_3$-Reshetikhin-Turaev invariant and if  $\partial M\neq \emptyset$, $RT_r(M)$ is a vector in $RT_r(\partial M)$.
\item If  $(M,\Sigma, \Sigma')$ is a cobordism from a surface $\Sigma$ to a surface  $\Sigma'$, then
$$RT_r(M):\ RT_r(\Sigma) \rightarrow RT_r(\Sigma'),$$ is a linear map such that compositions of cobordisms are sent to compositions of linear maps (up to powers of $q$).
\item The 3-manifold invariants $RT_r$ are multiplicative under disjoint union and for connected sums we have
 $$RT_r(M\# M')=\eta_r^{-1}RT_r(M)RT_r(M'),$$ 
 where $\eta_r={\displaystyle{ \frac{2\sin(\frac{2\pi}{r})}{\sqrt{r}}}}$. Furthermore we have $RT_r(S^2\times S^1)=1$.

\end{enumerate}
\begin{remark} \label{powers}
In this paper we will be concerned with the question of whether  the maps $RT_r(M)$ are invertible, and in the case we have inverses, we will study the  $r$-growth rate of the operator norm of the inverses.
Since these properties are not affected by multiplication by a power of $q$ in the sequel we will assume that compositions of cobordisms are sent to compositions of linear maps $RT_r$.
\end{remark}

The spaces  $RT_r(\Sigma)$ are certain quotients of  Kauffman  bracket skein modules (at $q$) of a handlebody bounded by $\Sigma$. 
In this paper we are interested in the case where $\Sigma$ is  the 2-torus $T^2$. In  \cite{BHMV2}, $T^2$ is viewed  as the boundary
 of a solid torus $D^2\times S^1$ and obtains elements $e_i \in RT_r(T^2)$ by taking the core  $\lbrace 0 \rbrace \times S^1$ of the solid torus decorated with the $(i-1)$-th
 Jones-Wenzl idempotent. For $r=2m+1$ and $q$ a $2r$-th root of unity, this process gives a family  $e_1,\ldots, e_{2m-1}$ of elements in  $RT_r(T^2)$
  \cite[Lemma 3.2]{BHMV2}. We have the following:

\begin{thm}\label{thm:basis}\cite[Theorem 4.10]{BHMV2} 
For $r=2m+1\geqslant 3$,
the Hermitian pairing of $RT_r(T^2)$ is positive definite and
 the family $e_1,e_2,\ldots, e_{m}$ is an orthonormal basis.
Moreover, for $0\leqslant i \leqslant m-1$, we have $e_{m-i}=e_{m+1+i}$.
\end{thm}

The Turaev-Viro invariants of compact oriented $3$-manifolds originally were defined as state sums over triangulations of manifolds (see \cite{TuraevViro}). In this paper however,  we will only use the relation between the Turaev-Viro invariants and Reshetikhin-Turaev invariants. This relation  was first proved by Roberts \cite{Roberts} in the case of closed $3$-manifolds, and extended to manifolds with boundary by Benedetti and Petronio  \cite{BePe}. We state it only in the case of manifolds with toroidal boundary, which is what we need,  and only for the root of unity we are working with in this paper:

\begin{thm}\label{thm:TV-RT}
Let $M$ be a compact oriented manifold with toroidal boundary, let $r\geqslant 3$ be an odd integer. Then,
 $$TV_r(M, e^{\frac{2\pi i}{r}})=||RT_r(M, e^{\frac{4\pi i}{r}})||^2,$$
where $|| \cdot ||$ is the natural Hermitian norm on $RT_r(\partial M)$ {\footnote{ Note that in  \cite[Theorem 3.7]{DKGromov} there is a typo and in the right hand side of the displayed formula $e^{\frac{\pi i}{r}}$ should be $e^{\frac{4\pi i}{r}}$}.}
In this paper, we are concerned with the case of $q=e^{\frac{2\pi i}{r}}$.
\end{thm}
Since for  $T^2$ a torus, the natural Hermitian form on $RT_r(T^2)$ is positive definite for any $q$, the invariants $TV_r(M)$ are non-negative.

\begin{remark} \label{normstuff} Given a finite-dimensional  Hermitian $\mathbb{C}$-vector space $V$, with a positive Hermitian pairing
$\langle . , . \rangle : V \times V \longrightarrow \mathbb{C}$, as above, we use $||.||$ to denote the norm induced by the Hermitian pairing (i.e. $||x||^2:= \langle x, x\rangle $).
Given a linear map $A:  V \longrightarrow V$, we will use $|||A|||$ to denote the norm of the operator, that is
$$|||A|||:= \underset{||x||=1}{\textrm{max}} ||A(x)||.$$
 We also define
$$ n(A):=\underset{||x||=1}{\textrm{min}} ||A(x)||.$$
If we assume that $A$ is invertible, then for any $x\in V,$ one has $$||x||=||A^{-1}(Ax)||\leq |||A^{-1}|||\cdot  ||Ax||,$$ with equality for some choice of $x.$ Inverting the inequality, one gets:  
$$ n(A)=|||A^{-1}|||^{-1}.$$
 
Finally, if $A$ is Hermitian, that is, $ \langle Ax,y \rangle=\langle x,Ay \rangle $ for any $x,y\in V$, then $A$ has an orthonormal basis of diagonalization, and we have
$$|||A|||=\underset{\lambda \in \mathrm{Spec}(A)}{\max}(|\lambda|)\ \ \  {\rm and} \ \ \ n(A)=\underset{\lambda \in \mathrm{Spec}(A)}{\min}(|\lambda|).$$
\end{remark}
\smallskip

\subsection{Genus one mutation} Consider $T^2\simeq \R^2/\Z^2$ as the quotient of $\R^2$ where $\pi_1(T^2)\simeq \Z^2$ acts by covering translations.
The elliptic involution $\iota$  on $T^2$ is defined by
$$\begin{array}{rcccl}
\iota & : & T^2\simeq \R^2/\Z^2 & \longrightarrow & T^2 \simeq \R^2/\Z^2
\\ & & (x,y) & \longrightarrow & (-x,-y)	.
\end{array}$$
and its isotopy class defines an element in the mapping class group of the torus $ \Gamma(T^2)$.

Given an element $\varphi$ in $ \Gamma(T^2)$ 
consider
$$M_{\varphi}:=[0,1]\times T^2 \underset{(1, x)\sim \varphi(x)}{\cup} T^2,$$
 the mapping cylinder of $\varphi$. Now  $RT_r(M_{\varphi})$ is a vector in $$RT_r(T^2)\otimes \overline {RT_r(T^2)}.$$
The latter space can be identified with $\mathrm{End}(RT_r(T^2))$ since we have  $RT_r(T^2)\simeq RT_r(T^2)^*$ through  the natural Hermitian form.  
The assignment  $\rho_r(\varphi)=RT_r(M_{\varphi})$ defines a projective representation
\begin{equation}\label{repres}
 \rho_r :   \Gamma(T^2) \longrightarrow \mathrm{End}(RT_r(T^2)). 
 \end{equation}

\begin{defin}
{\rm{ A compact oriented $3$-manifold $M'$ is said to be obtained from another compact oriented $3$-manifold $M$ by genus one mutation,  if $M'$ is obtained from $M$ by cutting along an embedded torus in $M$ and reglueing using the elliptic involution of $T^2$. }}
\end{defin}

We will make use of the following fact, which was proved by \cite{Rong94} for the $3$-manifold invariants, and which we state for cobordism invariants:

\begin{lem}
	\label{lemma:gen1mutation} If $C$ and $C'$ are two cobordisms, and $C'$ is obtained from $C$ by genus one mutation, then $RT_r(C)=RT_r(C')$ for any odd integer $r\geq 3$.
\end{lem}
\begin{proof}
	Let $T\simeq T^2$ be a torus embedded in $C$ such that $C'$ is obtained by cutting $C$ along $T$ and regluing using the elliptic involution. Equivalently, considering a regular neighborhood $N$ of $T$ in $C$, one can say that $C'$ is obtained from $C$ by replacing a trivial cylinder $N\simeq T\times [0,1]$ by the mapping cylinder 
$M_{\iota}$ of the elliptic involution $\iota$ on $T$. 
	Since the elliptic involution is in the kernel of the representation  $\rho_r$  (see for example \cite{FWW02}), the mapping cylinder of the elliptic involution and the trivial cylinder have the same image by $RT_r$. Moreover, $C\setminus N$ and $C'\setminus M_{\iota}$ are equivalent cobordisms. Since $RT_r$ is a TQFT, the maps $RT_r(C)$ and $RT_r(C')$ are obtained from $RT_r(C\setminus N)$ and $RT_r(T^2\times [0,1])$ (resp. $RT_r(C\setminus N)$  and $RT_r(M_{\iota})$)  by tensor contraction, and therefore are the same map. 
\end{proof}

\section{TQFT maps of Seifert fibered spaces}

Let $S=S(B; \ \frac{q_1}{p_1} \ldots  \frac{q_n}{p_n})$ denote the orientable   Seifert fibered $3$-manifold with fiber space a $2$-orbifold $B$,  and with singular fiber invariants $(q_1, p_1)\ldots (q_n, p_n)$ in the notation of  \cite{JN83}. Recall that, for $i=1, \cdots n$, $q_i$ is coprime to $p_i$ and that the integers $p_1, \cdots, p_n$ are called the multiplicities of the exceptional fibers. 
In particular, if $p_i=1$, for all $i=1, \cdots, n$, then  $S$ is an $S^1$-bundle over the surface $B$. If the surface $B$ has boundary then $S$ also has boundary which is a union of tori.
If $\partial S$ has two components,  say $T$ and $T'$, then as discussed in Section \ref{sec:prelim}, the  Reshetikhin-Turaev $\mathrm{SO}_3$-TQFT 
gives a linear map
$$RT_r(S):\ RT_r(T) \rightarrow RT_r(T'),$$
for any odd integer $r\geqslant 3$ and a primitive $2r$-th root of unity $q$. 
As earlier we use  $|| \cdot ||$ to denote   the norm induced by the Hermitian form on $RT_r(\partial S)$, and we use $||| \cdot |||$ to denote the operator norm of linear maps  between TQFT
spaces.  That is,  in the case that $RT_r(S)$ is invertible, we will have
$$||| RT_r(S)^{-1}|||:= \underset{||x||=1}{\textrm{max}} ||RT_r(S)^{-1}(x)||,$$
where $x\in RT_r(T')$.

Our main result in this section is the following:

\begin{thm} \label{invertible} For  $S=S(B; \ \frac{q_1}{p_1} \ldots  \frac{q_n}{p_n})$   a Seifert fibered 3-manifold, with $\partial S= T\cup T'$,
 the linear map $RT_r(S):\ RT_r(T) \rightarrow RT_r(T')$, is invertible for all odd $r$ coprime to  $p_1, \cdots, p_n$.
Furthermore, there are constants $C$ and $N>0$ such that
 $$||| RT_r(S)^{-1}||| \leqslant  C r^N.$$
\end{thm}

The last part of Theorem \ref{invertible} says that the operator norm of the inverse of $ RT_r(S)$ grows at most polynomially. Next we prove a couple of lemmas that we need for the proof of the 
theorem.

\begin{lem} \label{cable} If $S_p:=S_p(A; \ \frac{q}{p})$ fibers over an annulus $A$ with an exceptional fiber of multiplicity $p>1$, the linear map
$RT_r(S_p)$ is invertible for all $r$ coprime to  $p$. Furthermore,  the operator norm  $||| RT_r(S_p)^{-1}||| $ grows at most polynomially in $r$.
\end{lem}
\begin{proof}It is known that $S_p$ is a cable space and for such spaces the operators $RT_r(S_p)$ and the growth of their norm were explicitly
computed  by Kumar and Melby \cite[Theorem 1.7]{KuM2}.
\end{proof}

We will use $\Sigma_{g, n}$  (resp. $P_{g, n}$) to denote an orientable (resp. non-orientable), compact surface of non-orientable genus $g$ and $n$ boundary components.
The second lemma we need is the following:
\begin{lem}\label{Klein} Suppose that  $S$ is one of the following Seifert fibered 3-manifolds:
\begin{enumerate} [(a)]
\item The trivial $S^1$ bundle over a torus with two holes $ \Sigma_{1, 2} $.
\item The twisted $S^1$-bundle over the Klein bottle with two holes $P_{2, 2}$.
\item The  twisted $S^1$-bundle over the Mobius band  with one hole $P_{1, 2}$.
\end{enumerate}
Then,
the linear map
$RT_r(S)$ is invertible for all odd $r\geq 3$, and the operator norm  $||| RT_r(S)^{-1}||| $ grows at most polynomially in $r$.
\end{lem}
\begin{proof} Let $S:= S^1\times  \Sigma_{1, 2} $ and $\partial S=T\cup T'$.
For $r=2m+1$, the operator $RT_r(S)$ is exactly the operator  $K$ computed in \cite[Section 5.10]{BHMV2}.
It is self-adjoint since it is symmetric with respect to the orthonormal basis
$e_1, \cdots, e_m$ of $RT_r(T^2)$. The eigenvalues of $K$ have been computed  there and  were shown to be

$$\lambda_j:= \frac{(-r)}{(t^{2j} -t^{-2j})^2}=\frac{r}{4\sin^2(\frac{4\pi j}{r})} , \ \ \textrm{where} \ \  j=1, \cdots, m,$$
and $t=e^{\frac{2\pi i}{r}}$.

Since $\lambda_j \neq 0$, the operator is invertible  for all odd $r\geq 3$ and the eigenvalues of the inverse are
$\lambda_j^{-1}:= \frac{(t^{2j} -t^{-2j})^2}{(-r)}$.
Since the operator is self-adjoint, to bound  $||| RT_r(S)^{-1}|||, $ it is enough to bound the eigenvalues.
Since
$$| \lambda_j^{-1}|\leqslant  \frac{4}{r},$$
the result follows. This finishes the proof of (a).

Let us now prove (b). Let $S':=S^1 \widetilde{\times} P_{2,2}$, be  the twisted $S^1$-bundle over the Klein bottle with two holes.
We claim that $S'$ is obtained from $S^1\times \Sigma_{1,2}$ 
by genus one mutation. To see this, consider an orientation-reversing simple closed curve $\gamma$ on $P_{2,2}$.  Cutting $P_{2,2}$ along $\gamma$ and regluing by a diffeomorphism of $S^1$ that reverses the orientation, we get back $\Sigma_{1,2}$.  In the same way, cutting $S'$ along $S^1 \times \gamma$ and reglueing by the elliptic involution, we get $S$, which finishes the proof of the claim.
Now by Lemma  \ref{lemma:gen1mutation} we get $RT_r(S')=RT_r(S)$, and the desired conclusion follows from part (a). 

Next we prove (c). Let $S'':=S^1\widetilde{\times}P_{1,2 }$ be the twisted $S^1$-bundle over the Mobius band with one hole. 
Note that gluing two copies of $P_{1,2}$ together along a boundary component, one gets $P_{2,2}$.  We also 
have that the square of $S'',$ as a cobordism $T^2\longrightarrow T^2$,  satisfies  $S''\circ S''=S'$.  Therefore,  $$RT_r(S''\circ S'')=RT_r(S')=RT_r(S).$$ Since $RT_r(S)$ is self-adjoint, it is diagonalizable in a Hermitian basis of $RT_r(T^2).$ However, since the eigenvalues $\lambda_j$ of $RT_r(S)$ are all distinct, $RT_r(S'')$ must be diagonalizable in the same basis, and its eigenvalues $\mu_j$ are square roots of the $\lambda_j$'s. Note that the $\lambda_j$'s are positive. Hence, the $\mu_j$'s are all real and $RT_r(S'')$ is self-adjoint.  Therefore, we get the inequality
$$ |\mu_j|^{-1}\leqslant \frac{2}{\sqrt{r}},$$
which implies $|||RT_r(S'')^{-1}|||\leq  \frac{2}{\sqrt{r}}.$

\end{proof}
\vskip 0.01in
We are now ready to prove Theorem \ref{invertible}.

\begin{proof} Let $S=S(B; \ \frac{q_1}{p_1} \ldots  \frac{q_n}{p_n})$   be  a Seifert fibered 3-manifold, with $\partial S= T\cup T'$.
The surface $B$ can be cut, along a disjoint union  of simple closed curves $\Gamma$, into annuli each of which  contains exactly one orbifold point of the fibration, or two-holed tori or one-holed Mobius bands that contain no orbifold point. Moreover, the curves of $\Gamma$ may be chosen so that each  is separating in $B$.  The inverse image of $\Gamma$ under the Seifert fibration map
 $S\longrightarrow B$ is a collection $ {\mathcal T}$ of tori in $S$ each of which is vertical with respect to the fibration.
 By construction, for each  component $S_i$ of  $S\setminus {\mathcal T}$ there are the following possibilities:
 
 \begin{enumerate} [(a)]
\item $S_i$ fibers over an annulus with one exceptional fiber.

\item $S_i$ fibers over $ \Sigma_{1, 2} $ with no exceptional fibers. Hence,
it is the trivial $S^1$ bundle over  $ \Sigma_{1, 2} $.
\item $S_i$ fibers over the Mobius band with one hole and no exceptional fibers. Hence, it is the  twisted $S^1$-bundle over $P_{1, 2}$.
\end{enumerate}
  
  Moreover, since each curve of $\Gamma$ is separating, $S$ as a cobordism is a composition of the cobordisms $S_i.$
  Since $RT_r$ is a TQFT operator, and with the understanding of Remark \ref{powers}, we have
  
  $$RT_r(S)=RT_r(S_m)\circ \cdots  \circ RT_r(S_1)  \circ RT_r(S_0).$$

 By Lemma  \ref{Klein}, if $S_i$ is as in (b)-(c) above, then $RT_r(S_i)$ is invertible for all $r$. If  $S_i$ is as in  case (a), then by Lemma \ref{cable}, $RT_r(S_i)$ is invertible for all $r$ coprime to the multiplicity
 of the exceptional fiber. 
 It follows that, for all odd $r$ coprime to $p_1, \cdots, p_n$,  $RT_r(S)$ is invertible with inverse
 $$RT_r(S)^{-1}=RT_r(S_0)^{-1}\circ RT_r(S_1)^{-1} \circ \cdots    \circ RT_r(S_m)^{-1}.$$
 Also by Lemmas  \ref{Klein} and \ref{cable}, for $i=0, \cdots m$, the operator norm  of the inverses $||| RT_r(S_i)^{-1}||| $ grows at most polynomially in $r$.
 Since $||| \cdot |||$ is sub-multiplicative under composition of linear operators, it follows that
 there are constants $C$ and $N>0$ such that
 $$||| RT_r(S)^{-1}||| \leqslant  C r^N.$$
\end{proof}
\medskip
\section{ Gluing Theorems}
\subsection{Seifert cobordisms} In this section we prove the following theorem, which in particular implies Theorem \ref{seifertintro} of the Introduction.

\begin{thm}\label{seifertglue} Let $S=S(B; \ \frac{q_1}{p_1} \ldots  \frac{q_n}{p_n})$ be a Seifert fibered 3-manifold with at least two boundary components, and let $M$ be any 3-manifold  with toroidal boundary.
Then, for any 3-manifold $M'$ obtained by gluing $S$  along a component of the boundary $T'\subset \partial S$,  to a component of $\partial M$,
there exist constants $A$ and $K>0$ such that
$$\frac{r^{-K}}{A} TV_r(M) \leqslant TV_r(M') \leqslant Ar^K TV_r(M).$$
Here, the upper inequality holds for all odd $r$, while the lower inequality holds for
all  $r$ coprime to $p_1, \cdots, p_n$.

 In particular, if the limsup in the definition of $LTV(M)$ is actually a limit, then $LTV(M')=LTV(M).$
\end{thm}

By the classification theorem of manifolds that admit Seifert fibrations, manifolds with more than two boundary components admit unique such fibrations. Hence the integers
$p_1, \cdots, p_n$ are uniquely determined by the  3-manifold and vice versa. Setting $I:=\{p_1, \cdots, p_n\}$,  we obtain Theorem \ref{seifertintro}.
\smallskip

\subsection{Invertible cable spaces} For the proof of Theorem \ref{seifertglue} we need to recall the notion of an invertible cable space from \cite{DKGromov}.
\begin{defin}{\rm {Let $S$ be a 3-manifold with toroidal boundary, with a distinguished torus boundary component $T$, and such that $\partial S$ has at least three boundary components.
$S$ 
 is called an invertible cable  space if (i) $S$ has zero simplicial volume (i.e. $\mathrm{Vol}(S)=0$);
and (ii) there is a Dehn filling along some components of $\partial S$ distinct from $T$ that produces a 3-manifold  homeomorphic to $T \times [0,1]$.}}
\end{defin}

\begin{cor}\label{invertiblecable} \cite[Corollary 8.3]{DKGromov} Let $M$ be a 3-manifold with toroidal boundary and $S$ be an invertible cable space.
Let $M'$ be obtained by gluing a component of $\partial S\setminus T$ to a component of $\partial M$.  Then,
there exist constants $A$ and $K>0$ such that
$$TV_r(M) \leqslant TV_r(M') \leqslant Ar^K TV_r(M)$$
for all odd  $r\geq 3$.
In particular, we have
$LTV(M)=LTV(M')$.
\end{cor}
 
 We will need the following:
\begin{lem}\label{lem:product} For any  $n\geq 3$,  $S:=S^1\times \Sigma_{0, n}$ is an invertible cable space.
\end{lem}
\begin{proof}
Since $S$ is an $S^1$-bundle, $\mathrm{Vol}(S)=0$.
Designate one component $T\in \partial S$ as the  distinguished component. Now, the trivial Dehn filling along  all but one of the tori in $\partial S\setminus T$,
produces the trivial $S^1$- bundle over an annulus, that is $S^1\times S^1\times [0, 1]$, where $T=S^1\times S^1$.
\end{proof}
\vskip 0.03in
\subsection{Proof of Theorem \ref{seifertglue}} 
By construction, 
$M'$ is obtained by gluing   a boundary component of  $S$ to a component of $\partial M$.
Since $S$ is a Seifert fibered manifold,  by \cite[Theorem 5.2]{DKGromov} (and its proof) we have
 
 $$ TV_r(M') \leqslant TV_r(S)\cdot  TV_r(M),$$
 for all odd integers $r\geq 3$. On the other hand, since $S$ is a Seifert fibered 3-manifold,  by
  \cite[Theorem 5.2]{DKGromov}, there are constants $A$ and $N>0$ such that
 $TV_r(S) \leqslant Ar^N$, for all odd integers $r\geq 3$. Thus, we have  $LTV(S) \leqslant  0$.
 Hence, the upper inequality in the statement of the theorem follows, and in particular, we have $$LTV(M') \leqslant LTV(M).$$
 \vskip 0.03in
 
 For the proof of the lower inequality we will distinguish two cases:
   \vskip 0.06in

  \noindent {\em {Case 1}.} Suppose that $S$ has exactly two boundary components, say $T$ and $T'$.
 By Theorem \ref{invertible}  the linear map $RT_r(S):\ RT_r(T) \rightarrow RT_r(T')$, is invertible for all odd $r$ coprime to  $p_1, \cdots, p_n$.
Furthermore, there are constants $C$ and $N>0$ such that
 \begin{equation} \label{poly}
 ||| RT_r(S)^{-1}||| \leqslant  C r^N.
 \end{equation}

If $M$ has only one boundary component, then $RT_r(M)$ is a vector in $RT_r(T^2)$, and
by the TQFT properties, we have that $$RT_r(M')=RT_r(S)(RT_r(M)).$$

  Now we can write $RT_r(S)^{-1}(RT_r(M'))=RT_r(M),$ and hence
 $$ ||RT_r(M)||\leqslant |||RT_r(S)^{-1}||| \cdot ||RT_r(M')||,$$
which in turn gives
$$
|||RT_r(S)^{-1}|||^{-1} \cdot ||RT_r(M)||\leqslant ||RT_r(M')||.
$$
The last inequality combined with (\ref{poly}), and after adjusting the constants, gives
\begin{equation}\label{adjust}
\frac{r^{-N_1}}{C} TV_r(M) \leqslant TV_r(M').
\end{equation}
Finally, after a further adjustment of  the constants  we get the desired result.
\vskip 0.05in

If $M$ has more than one boundary component, let $T_1$ denote the one that is used to glue $S$.
Then, $M'$ may equivalently be seen as obtained from $M$ by gluing the cobordism $$S':=S \coprod (\partial M \setminus T_1)\times [0,1]$$ onto $\partial M$.
 The latter is a cobordism $\partial M \rightarrow \partial M$, and by the TQFT properties, $RT_r(S')$ is invertible.
 
 We claim that $|||RT_r(S')^{-1}|||=|||RT_r(S)^{-1}|||$.  Indeed, $RT_r$ is a monoidal functor, so $$RT_r(S')=RT_r(S)\otimes \mathrm{id}_{RT_r(\partial M \setminus T_1)}.$$
  Moreover, the operator norm $||| \cdot |||$ is multiplicative under tensor product of Hermitian vector spaces and maps. The rest of the argument proceeds exactly as before.
  \vskip 0.04in

  \noindent {\em {Case 2}.} Suppose that $S$ has $n\geq 3$ boundary components and $k$ exceptional fibers of orders $p_1,\ldots,p_k.$ Pick a curve $\gamma$ in the orbifold $B$ that separates it into a surface $\Sigma_{0,n}$ (containing no orbifold point) and an orbifold $B'$ with exactly two boundary components. We can furthermore assume that the boundary component of $S$ glued onto $M$ corresponds  to a curve in $B'.$ Taking the pre-images under the Seifert fibration, we see that $M'$ is obtained from $M$ by first gluing a Seifert manifold with two boundary components on a torus boundary component of $M,$ obtaining a $3$-manifold $M_0$, and  then gluing $S^1\times \Sigma_{0,n}$ on a boundary component of $M_0.$
  
  By Case 1, there exist constants $A,K>0$ such that
  $$\frac{1}{Ar^K}TV_r(M)\leqslant TV_r(M_0) \leqslant Ar^K TV_r(M)$$
  for any $r$ coprime to all of the integers $p_i.$
  By Corollary \ref{invertiblecable}, there exist constants $B,L>0$ such that
   $$
   \frac{1}{Br^L}TV_r(M_0)\leqslant TV_r(M') \leqslant Br^L TV_r(M_0)
   $$
   for any odd $r\geq 3$. Therefore, we get the desired inequalities in this case as well.\qed
   
\subsection{Plumbed  Cobordisms}
A \emph{graph manifold} $G$ is a 3-manifold that can be decomposed into Seifert fibered spaces by cutting along a collection ${\mathcal T}$ of incompressible tori. To any graph manifold $G$ we associate a graph $T(G)$
whose  vertices correspond to components of $G\setminus {\mathcal T}$ and whose  edges correspond to the  tori  in ${\mathcal T}$ along 
which those components are glued.
The leaves are vertices of valence one.
If the graph above is a tree, we will say that $M$  is a  \emph{plumbed manifold}.
The following generalizes Theorem \ref{seifertglue}  to plumbed 3-manifolds.

\begin{cor}\label{graphglue}Let $G$ be a plumbed 3-manifold  such that each leaf on $T(G)$
has  at least one boundary component coming from a component of $\partial G$.  Let
$M$ be any 3-manifold  with  non-empty toroidal boundary.
Then, for any 3-manifold $M'$ obtained by gluing $G$  along a component $T'\subset \partial G$ to a component of $\partial M$,
there exist constants $A$ and $K>0$  and a finite set  $I$ of non-zero integers such that
$$
\frac{r^{-K}}{A} TV_r(M) \leqslant TV_r(M') \leqslant Ar^K TV_r(M).$$
Here, the upper inequality holds for all odd integers $r>2$ and the lower inequality holds for all $r$
not divisible by any of the numbers in $I$.

 In particular, if the limsup in the definition of $LTV(M)$ is actually a limit, then $LTV(M')=LTV(M)$.
\end{cor}
\begin{proof} The proof is by induction on the number of edges of $T(G)$.
If there are no edges,  the conclusion follows from Theorem \ref{seifertglue}. In this case,  the set $I$ is the set of multiplicities
of the exceptional fibers of $G$.
Otherwise,  remove from $T(G)$ an edge $e$  that ends  at a leaf $S$ to obtain a tree $T(G_1)$,
associated to a graph manifold $G_1$. 
By hypothesis $S$ has a boundary component which comes from $\partial G$. 
Suppose that the edge $e$ corresponds to  
a torus $T'\in { \mathcal T}$. Now cutting $G$ along $T'$ we obtain two 3-manifolds: One is the graph manifold $G_1$ above and the second is 
the Seifert fibered manifold $S$, which by hypothesis has at least two boundary components. Let $M_1$ denote the 3-manifold obtained by gluing $G_1$ 
to $M$ along the component of $\partial G_1$ that corresponds to the component of $\partial G$ glued along $\partial M$ in the construction of $M'$.
Now $M'$ is obtained by gluing $S$ to $M_1$ along a boundary component.
The result follows by applying the induction hypothesis to $M$ and $G_1$ and Theorem \ref{seifertglue} to $M_1$ and $S$.
\end{proof}

\section{Volume Conjecture applications}
In this section, we discuss applications of Theorems  \ref{seifertglue}  and Corollary \ref{graphglue} to 
Conjecture \ref{TVvolumeconj}.  Recall that the Gromov norm (and hence simplicial volume)
is subadditive under gluing 3-manifolds along tori and additive if gluing along incompressible tori. It is also additive
under the operation of gluing 3-manifolds along essential 2-spheres. For details, see \cite{ThurstonGT3manifolds}.

\begin{thm}\label{vc} Let $M$ be a 3-manifold with non-empty boundary such that
\begin{equation}\label{limit}
\underset{r \rightarrow \infty, \ r \ \textrm{odd}}{\lim} \frac {2\pi}  {r} \log |TV_r(M)|=\mathrm{Vol}(M).
\end{equation}
Suppose that $M'$ is obtained from $M$ by gluing to a component $T'\subset \partial M$, either
\begin{enumerate} [(a)]
\item a Seifert fibered space $S$ as in Theorem \ref{seifertglue}; or 
\item a plumbed 3-manifold $G$ as in Corollary \ref{graphglue}. 
  \end{enumerate}
Then $\mathrm{LTV}(M')=\mathrm{Vol}(M')$.
\end{thm}
\begin{proof} By Theorem \ref{seifertglue}, we have $LTV(M')=LTV(M)$ in the case of (a) and by Corollary  \ref{graphglue}, we have $LTV(M')=LTV(M)$ in the case of (b). 
So in both cases we only need to prove that $\mathrm{Vol}(M')=\mathrm{Vol}(M)$. We will discuss the details for (a). The proof of (b) is completely analogous.

The manifold $M'$ is the gluing of $M$ and the Seifert manifold $S$ along a torus $T$. Since $S$ has at least two boundary components, in particular it is not a solid torus and the torus $T$ is incompressible in $S$. There are thus two cases:
	
\textit{Case 1:} The torus $T$ is also incompressible in $M$. Then the torus $T$ is also incompressible in $M'$ and we have $$\mathrm{Vol}(M')=\mathrm{Vol}(M)+\mathrm{Vol}(S)=\mathrm{Vol}(M),$$ since the simplicial volume is  additive under gluing along an incompressible torus, and $S$ is a Seifert manifold.

\textit{Case 2:} The torus $T$ is compressible in $M$.  
Then $M$ is the connected sum 
of a solid torus $V$ and another 3-manifold $M_0$.
Since the simplicial volume is additive under disjoint union and under taking connected sums, we have $\mathrm{Vol}(M_0)=\mathrm{Vol}(M)$. 
Now we can obtain $M'$ as a connected sum $M'=M_0\# S'$ where $S'$ is obtained by gluing $S$ to the solid torus $V$. Since $S'$ is a Seifert fibered manifold, we have $\mathrm{Vol}(S')=0$.
Again by  the additivity of the simplicial volume under connected sums and disjoint unions, we have
$$\mathrm{Vol}(M')=\mathrm{Vol}(M_0)+ \mathrm{Vol}(S')=\mathrm{Vol}(M_0)=\mathrm{Vol}(M),$$
giving the desired result.
\end{proof}
\smallskip

\begin{remark}{\rm{
The lower inequality of Theorem
\ref{seifertglue} holds only for values of $r$ that are coprime  to the multiplicities of all the singular fibers of $S$. Indeed, \cite[Theorem 1.7]{KuM2} shows that in the case that 
 the Seifert space $S$ fibers over an annulus with one exceptional fiber of multiplicity $p$, then the TQFT operator  $RT_r(S)$ is not invertible if $r$ is not coprime
to $p$. This is primarily the reason for requiring that
$$\underset{r \rightarrow \infty, \ r \ \textrm{odd}}{\lim} \frac {2\pi}  {r} \log |TV_r(M)|=\mathrm{Vol}(M),$$
in the statement of Theorem \ref{vc} above. Note however,  that the proof also goes through if we weaken the hypothesis  on $M$ to
taking the limit over all odd primes.}}
\end{remark}

\begin{remark}{\rm{
 Note that if $M$ in Theorem \ref{vc} has zero simplicial volume the conclusion of the theorem follows if the limit in Equation (\ref{limit}) is replaced by limsup. Indeed, Theorem
 \ref{seifertglue}  and Corollary \ref{graphglue} imply that if  $\mathrm{LTV}(M)=0$ then $\mathrm{LTV}(M')=0$. On the other hand, since the Gromov norm is subadditive under gluing 3-manifolds along tori
 gluing manifolds of simplicial volume zero produces volume zero manifolds. }}
  \end{remark}

Next we have two  results that
prove the volume conjecture for
Seifert fibered 3-manifolds with non-empty boundary and for large classes of graph manifolds. Our first result is the following:

\begin{named}{Corollary \ref{VCS}} Suppose that $S$ is  an oriented Seifert fibered 3-manifold that either has a non-empty boundary, or it is closed and admits an orientation-reversing involution. Then we have
$$LTV(S)=\underset{r \rightarrow \infty, \ r \ \textrm{odd}}{\limsup} \frac {2\pi}  {r} \log |TV_r(S)|=\mathrm{Vol}(S)=0.$$
\end{named}
\begin{proof}
First suppose that $\partial S\neq \emptyset$.
Removing from $S$ the neighborhood of a regular fiber of $S$, which is a solid torus $M:=D^2\times S^1$, we obtain a Seifert manifold $S'$
that has at least two boundary components and one of them will be glued to  $\partial M$.
On the other hand, by Theorem \ref{thm:TV-RT}, we have
$$TV_r(M)=TV_r(D^2\times S^1)=RT_r(S^2\times S^1)=1,$$
and hence we obtain $\textit {LTV}(M)=\mathrm{Vol}(M)=0$. Now the result follows by part (a) of Theorem \ref{vc}.

Next suppose that $S$ is closed and there is an orientation-reversing involution $i: S\longrightarrow S$. Then $S$ is the double  of a Seifert fibered manifold $S_1$, where $S_1$ has non-empty boundary.
That is, if we let $\bar{S_1}$ denote $S_1$ with the opposite orientation, then $S$ is obtained by identifying $\bar{S_1}$ and  $S_1$  along their boundary.
On one hand we have  $\mathrm{Vol}(S)=0=\mathrm{Vol}(S_1)$.
On the other hand, by Theorem \ref{thm:TV-RT},
$$TV_r(S,e^{\frac{2\pi i}{r}})=||RT_r(S_1, e^{\frac{4\pi i}{r}})||=TV_r(S_1, e^{\frac{2\pi i}{r}}),$$
and hence $LTV(S)=LTV(S_1)=0$.

\end{proof}

Now we turn to the second result that considers plumbed 3-manifolds.
\begin{cor}\label{Graph} Let $G$ be a plumbed manifold with non-empty boundary and with an associated tree $T(G)$. Suppose that all but at most one of the leaves of $T(G)$ correspond to  3-manifolds with at least one boundary component coming from $\partial G$.
Then,
$$\textit {LTV}(G) =\mathrm{Vol}(G)=0.$$

\end{cor}

\begin{proof}  Remove from $G$  the neighborhood of a regular fiber of a leaf $S$. Then proceed as in the proof of Corollary \ref{VCS} using part (b) of Theorem \ref{vc}.
\end{proof}

\begin{remark} Some cases of  Corollary \ref{VCS} were also verified by the third author of this paper using  different methods \cite{Shashini}.
Note that in this paper we establish $\textit {LTV}(S)=0$  by restricting to a sequence of integers  $r\to \infty$
that are coprime to the multiplicities of the exceptional fibers of the Seifert fibration.
In contrast to that, \cite{Shashini} uses a sequence of integers $r\to \infty$, where $r$ is divisible by the multiplicities of all fibers, to establish the volume conjecture for some families of Seifert fibered manifolds.
For example, \cite{Shashini} shows if $S=S(B; \ \frac{q_1}{p_1} \ldots  \frac{q_n}{p_n})$  is a Seifert fibered space with boundary and $p_1,\cdots, p_n$ are odd and  pairwise coprime, then
$$\underset{r \rightarrow \infty, \ \textrm{p divides r}}{\lim} \frac {2\pi}  {r} \log |TV_r(S)|=\mathrm{Vol}(S)=0,$$
where $p:=p_1\ldots p_n$.

\end{remark}
\medskip

\section{Hyperbolic cobordisms}
In view of our results here it is reasonable to ask what is  the behavior of the TQFT operator maps for  cobordisms with non-zero simplicial volume.
For example,  let $M$ be a 3-manifold with two torus boundary components, $\partial M= T\cup T'$ whose interior admits a hyperbolic structure.  As before we get 
operators $RT_r(M):\ RT_r(T') \rightarrow RT_r(T)$. If   $RT_r(M)$ is invertible, one would hope that the operator norm $||| RT_r(M)^{-1}|||^{-1}$ grows exponentially as $r\to \infty$. 
However, as we will see below this is not always the case.

With $M $ as above, on  the torus $T'\subset \partial M$  take  a simple closed curve representing slope ${\bf s}$, and let $M({\bf s})$ denote the 3-manifold obtained by Dehn filling $M$ along ${\bf s}$.
If the  length of the geodesic representing ${\bf s}$ on $T'$ is large enough, then $M({\bf s})$ is also hyperbolic \cite{ThurstonGT3manifolds}.
However, for slopes represented by shorter geodesics the resulting manifold can be exceptional (i.e.  non-hyperbolic) and in particular  $M({\bf s})$ can be a Seifert fibered 3-manifold.
For example,  if $M$ is the complement of the Whitehead link in $S^3$ then a Dehn filling along one of the components of $\partial M$ produces a solid torus which has volume $0.$
The next proposition shows that in these cases the operator norm $||| RT_r(M)^{-1}|||^{-1}$  grows at most polynomially.

\begin{prop}\label{notexpo}Let $M$ be a cobordism from $T$ to $T'$ as above, and suppose the  map $ RT_r(M)$ is invertible.
Suppose that $M$ admits a Dehn filling with slope $s$ along a component of  $\partial M$ so that $M(s)$ is a 3-manifold of zero simplicial volume. 
Then,   $||| RT_r(M)^{-1}|||^{-1}$ grows at most polynomially.
\end{prop}
\begin{proof}

By assumption $M({\bf s})$ has boundary a single torus $T$ and $RT_r(M({\bf s}))$ is a vector in $RT_r(T)$.
Since $M({\bf s})$ has zero simplicial volume, 
 by \cite[Theorem 11]{DKGromov} its norm with respect to the  Hermitian pairing on $RT_r(T)$, denoted by 
$||RT_r(M({\bf s}))||$, grows at most polynomially in $r$.

By Remark   \ref{normstuff} we have
\begin{equation} \label{norm1}
||| RT_r(M)^{-1}|||^{-1}=\underset{||x||=1}{\textrm{min}} ||RT_r(M)(x)||,
\end{equation}
where $x\in RT_r(T')$. 
On the other hand, by the TQFT properties, 
\begin{equation} \label{norm2}
RT_r(M({\bf s}))=RT_r(M)(e_r({\bf s})),
\end{equation}
 where $e_r({\bf s}) \in RT_r(T')$ is the vector the TQFT-functor $RT_r$ assigns to the solid torus where the meridian is the curve representing the slope ${\bf s}$.
We claim that $e_r({\bf s})$ is a vector of Hermitian norm 1. Indeed, $e_r({\bf s})$ is the $RT_r$-vector of a solid torus $D^2\times S^1$ but with the meridian of $D^2\times S^1$ identified with the curve of slope ${\bf s}$ on $T^2.$ Hence, it is the image of the basis vector $e_1$ introduced in Theorem \ref{thm:basis} by $\rho_r(\phi_{\bf s}),$ where $\rho_r$ is the quantum representation of  the mapping class group of $T^2$ (see Equation (\ref{repres}) in Section 2)
 and $\phi_s$ is any mapping class that sends the meridian of $T^2$ to the curve of slope ${\bf s}$. Since the image of the quantum representation $\rho_r$ consists only of unitary maps, $e_r({\bf s})$ has norm $1$. 

Now Equations (\ref{norm1}) and  (\ref{norm2}) and the discussion in the beginning of the proof imply
\begin{equation}\label{exponential} ||| RT_r(M)^{-1}|||^{-1}\leqslant ||RT_r(M({\bf s}))||\leqslant A\cdot r^N,
\end{equation}
for some constants $A$, $N>0$.
\end{proof}

\medskip

Note that the first part of inequality (\ref{exponential}) implies that if $||| RT_r(M)^{-1}|||^{-1}$ grows exponentially with $r$, then the invariants
$$TV_r(M({\bf s}), e^{\frac{2\pi i}{r}})=||RT_r(M({\bf s}), e^{\frac{4\pi i}{r}})||^2,$$ grow exponentially. Tools that allow  one to establish exponential growth of the Turaev-Viro invariants of Dehn fillings are highly desirable as they
will lead to progress on the volume conjecture as well as on another important conjecture in quantum topology; the AMU conjecture \cite{DKAdvances, DKIndiana, KMe}.
We ask the following question:

\begin{pro} Construct examples of hyperbolic cobordisms $M: T\longrightarrow T'$ such that   $RT_r(M):\ RT_r(T) \rightarrow RT_r(T')$ is invertible and 
$||| RT_r(M)^{-1}|||^{-1}$ grows exponentially with $r$.
\end{pro}

In view of Proposition \ref{notexpo} one has to look at hyperbolic  cobordisms $M: T\longrightarrow T'$, such that all the 3-manifolds by filling one of the components of $\partial M$ have non-zero simplicial volume.
 One way to obtain such cobordisms is to consider complements  of two-component highly twisted links in $S^3$ (see \cite{FKP} and references therein). In these cases all the Dehn fillings of either of the two 
boundary components produce hyperbolic 3-manifolds.
\smallskip

{\bf {Acknowledgement.}} {The research of E.K. and S.M. is partially supported by NSF grants  DMS-2304033 and DMS- 2603192.  The research of R.D. is partially
	supported by the ANR project ``NAQI-34T” (ANR-23-ERCS-0008) and by the
	project ``CLICQ” of the R\'egion Bourgogne-Franche Comt\'e}.

\bibliographystyle{abbrv}

\bibliography{biblio}

\end{document}